\documentclass[12pt,reqno]{amsart}
\usepackage[usenames,dvipsnames]{color}
\usepackage{graphicx}
\usepackage{fullpage}
\usepackage{float}
\usepackage{wrapfig}
\restylefloat{figure}
\usepackage{amsthm}
\usepackage{enumitem}
\usepackage{pgf,tikz}
\usepackage{pgfplots}
\usetikzlibrary{decorations.markings}
\usetikzlibrary{arrows,shapes,positioning,calc,shadows}
\usepackage{float}
\usepackage[mathscr]{euscript}
\usepackage[colorlinks=true,citecolor=cyan,urlcolor=blue]{hyperref} 
\usepackage{wrapfig}
\usepackage{amssymb}
\usepackage{amsmath}

\def\pref#1{{\rm (\ref{#1})}}

\def\bleu{\textcolor{blue}}

\def\Z{\mathbb{Z}}
\def\N{\mathbb{N}}
\def\R{\mathbb{R}}
\def\C{\mathbb{C}}
\def\H{\mathbb{H}}
\def\Im{\mathrm{Im}}
\def\Re{\mathrm{Re}}
\def\PSL{\mathrm{PSL}}

\newcommand\minus{%
  \setbox0=\hbox{-}%
  \vcenter{%
    \hrule width\wd0 height \the\fontdimen8\textfont3%
  }%
}

\def\define#1{\bleu{\bf #1}}

\makeatletter
\newtheorem*{rep@theorem}{\rep@title}
\newcommand{\newreptheorem}[2]{
\newenvironment{rep#1}[1]{
\def\rep@title{#2 \ref{##1}}
\begin{rep@theorem}}
{\end{rep@theorem}}}
\makeatother

\newtheorem{theorem}{\bleu{Theorem}}
\newtheorem{lemma}{\bleu{Lemma}}

\newtheorem{corollary}{\bleu{Corollaire}}
\newtheorem{proposition}{\bleu{Proposition}}

\numberwithin{equation}{section}
\numberwithin{theorem}{section}
\numberwithin{lemma}{section}
\numberwithin{prop}{section}
\numberwithin{corollary}{section}

\def\segment#1#2#3#4
{\draw [shift={(#1,0.)},color=blue] plot[samples=100,domain=#3:#4,variable=\t]({#2*cos(\t r)},{#2*sin(\t r)});}
\def\segmentcol#1#2#3#4#5
{\draw [shift={(#1,0.)},color=#5] plot[samples=100,domain=#3:#4,variable=\t]({#2*cos(\t r)},{#2*sin(\t r)});}
\def\Segmentcol#1#2#3#4#5
{\draw [shift={(#1,0.)},color=#5,line width=1.5pt] plot[samples=100,domain=#3:#4,variable=\t]({#2*cos(\t r)},{#2*sin(\t r)});}

\title{Golden Ratio and Phyllotaxis,\\
what is the mathematical link?}

\author{Fran\c{c}ois Bergeron}
\address{Fran\c{c}ois Bergeron,
D\'epartement de math\'ematiques, Universit\'e du Qu\'ebec \`a Montr\'eal}
\email{Bergeron.Francois@uqam.ca}

\author{Christophe Reutenauer}
\address{Christophe Reutenauer,
D\'epartement de math\'ematiques, Universit\'e du Qu\'ebec \`a Montr\'eal}
\email{Reutenauer.Christophe@uqam.ca}

\date{\today}
\begin{document}
\begin{abstract}
Exploiting Markoff's Theory for rational approximations of real numbers, we explicitly link how hard it is to approximate a given number to an idealized notion of growth capacity for plants which we express as a modular invariant function depending on this number. Assuming that our growth capacity is biologically relevant, this allows us to explain in a satisfying mathematical way why the golden ratio occurs in nature.
\keywords{Modular Group \and Markoff Approximation Theory \and Golden Ratio\and Phyllotaxis}
\end{abstract}

\maketitle
 \parskip=0pt
{ \setcounter{tocdepth}{1}\parskip=0pt\footnotesize \tableofcontents}
\parskip=8pt  
\parindent=20pt

\section{Introduction}\label{intro}

 \begin{wrapfigure}[8]{L}{0.25\textwidth}
 \small{
\begin{center}\vskip-10pt
 \includegraphics[width=40mm]{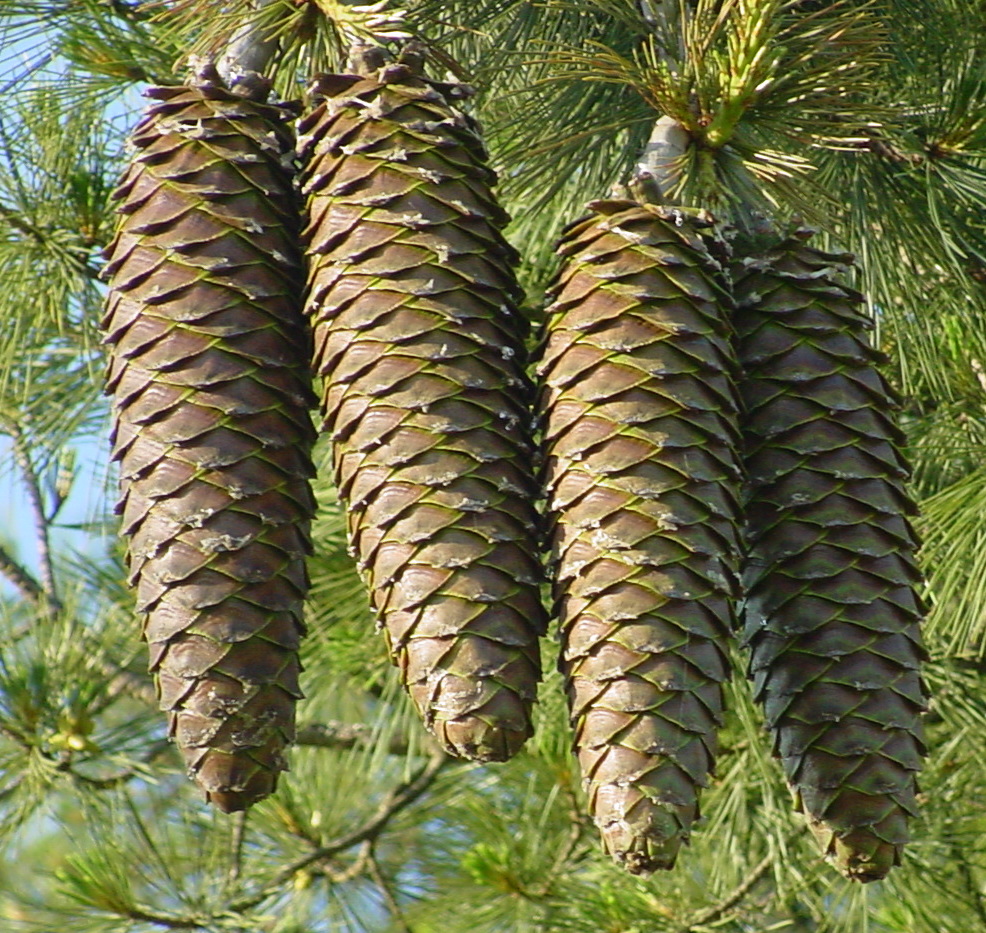} 
 \label{fig:sugarpine}
 \end{center}}
 \quad{\small Fig.~1 {Sugar pine}}
 \end{wrapfigure}
  \setcounter{figure}{1}
Among the frequently mentioned mathematical notions that occur in natural phenomena, surely Fibonacci numbers $F_n$:
\smallskip 

 \qquad   $F_0=1,\ F_1=2,\ F_2=3,\ F_3=5,\ F_5=8,\ \ldots$,
 \smallskip 

\noindent with $F_n=F_{n-1}+F_{n-2}$, and the golden ratio $\varphi=(1+\sqrt{5})/2$, rate close to the very top in the broad public media.  It is perhaps both the simplicity of their definition and their ties to beautiful patterns (such as the photo\footnote{Photo:  Richard Sniezko - US Forest Service.} in Figure~\ref{fig:sugarpine}) that make them especially appealing to a general audience. This fascination for the interplay between Fibonacci numbers and nature apparently goes back at least to Kepler, with some earlier allusions by Da Vinci.
It is also tantalizing that they are nicely related by the fact that quotient of successive Fibonacci numbers are the ``best'' rational approximations of the golden ratio
\begin{displaymath}
    \frac{1+\sqrt{5}}{2}\ \simeq\  \frac{F_{n+1}}{F_n},
\end{displaymath}
so that their joint story has both aesthetic appeal and some intellectual surprise. 
In many social occasions, mathematicians are at risk of being asked for an explanation of why the golden ratio and Fibonacci numbers should play such a nice role, often placing them in somewhat of a quandry since some part of the answer must necessarily involve an understanding of some biologico-physical law that underlies the phenomenon considered. Indeed, it stands to reason that some optimization of an advantageous trait must be behind the appearance of the patterns observed, as the cornered mathematician is bound to try to underline. If more knowledgeable about it, he/she may underline that the golden ratio is characterized by the fact that it is one of the hardest (we will explain how below) numbers to approximate by a rational number, and that this must be why it occurs in the alleged optimization involved. Albeit, this is somewhat incomplete since no explicit tie is established between a biological law and the mathematical fact referred to.

Our objective in this paper is to explain how to directly link this notion of ``hard to approximate'' to one of the abstract models of plant growth considered by some \href{https://en.wikipedia.org/wiki/Phyllotaxis}{phyllotaxis} researchers (see~\cite{iterson,Okabe2,RZ}). In fact, there is a lot of literature and interesting work (see~\cite{A}) pertaining to mathematical aspects of phyllotaxis, and a very nice broad historical overview of the plentiful and varied efforts along these lines may be found in \cite{A}. Noteworthy from our perspective are the more recent work of  \cite{AGH,D,Leigh,MK}, and the hard to approximate justification is mentioned in some outreach texts such as \cite{Du,RZ}. In \cite{AGH} is given a rigorous mathematical analysis of a model of plant pattern formation from the point of view of dynamical systems, explaining the occurrence of Fibonacci numbers in terms of fixed points and bifurcation patterns. 
Notwithstanding this, we could not find in the literature a truly satisfying direct mathematical link between the hard to approximate property of the golden ration and some abstract mode of growth of plants, with a precise mathematical formulation of the nature of this direct tie.  This work does propose such a formulation, but we make no claim that our model has been validated from the point of view of Biology. We leave this to be checked by the experts in the field.

\begin{wrapfigure}[9]{L}{0.26\textwidth}
\begin{center}\vskip-5pt
 \includegraphics[width=35mm]{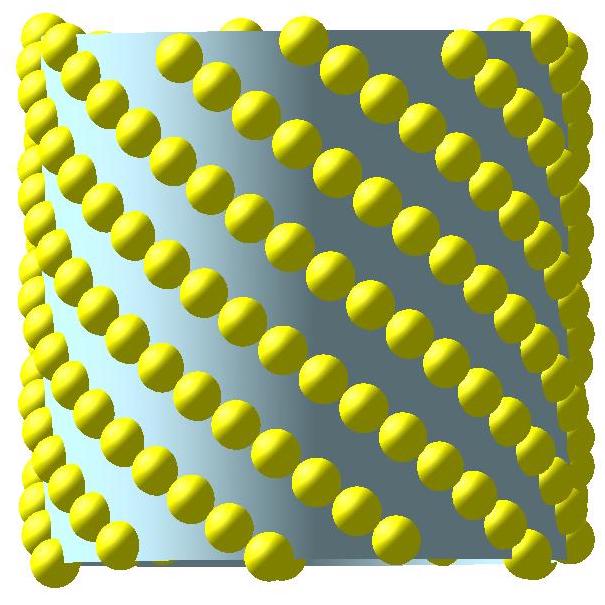}
 \vskip-10pt
\caption{\small{Cylindric plant}} \label{fig:cylindric}
 \end{center}
 \end{wrapfigure}
   \setcounter{figure}{2}
   We start by recalling how the notion of hard to approximate by a rational number has been beautifully developed by Markoff\footnote{This is the same Markov as in the well-known Markov chains theory; who used this surname spelling in his French publications.} in two seminal papers~\cite{markoff1,markoff2} that appeared in 1879 and 1880.  His theory is nicely presented in a recent book of Aigner \cite{Ai}, where more details may be found. Following Markoff's tack, we associate to each irrational number $x$ its \href{https://en.wikipedia.org/wiki/Lagrange_number}{\define{Lagrange number}}, denoted $L=L(x)$. This is the largest (supremum of the set of) real number such that there are infinitely many rational approximations $p/q$ of $x$ for which we have the inequality
    \begin{displaymath}\left| x-\frac{p}{q}\right| < \frac{1}{L\,q^2}.\end{displaymath}
Part of Markoff's Theory says that $L(x)=\sqrt{5}$, if $x$ is equivalent\footnote{Here, a number is considered to be \define{equivalent} to the golden ratio if its continued fraction expansion only contains $1$ after a certain rank.} to the golden ratio; and that $L(x)\geq\sqrt{8}$ for any other real number. In other words, any number $x$, not equivalent to the golden ratio, affords infinitely many rational approximations for which
    \begin{displaymath}\left| x-\frac{p}{q}\right| < \frac{1}{\sqrt{8}\,q^2},\end{displaymath}
 whereas this is not so for the golden ratio. 
 It is in this precise sense that  the golden ratio (and its equivalents) is considered hardest to approximate. Markoff's Theory goes on to give a very nice filtration of real numbers with respect to how easier they become to approximate, once some relevant subsets are removed. He shows that there is a sequence of Lagrange numbers $L_n$, generalizing  $\sqrt{5}$ and $\sqrt{8}$ above, of the form 
     \begin{displaymath}L_n=\sqrt{9-\frac{4}{m_n^2}},\end{displaymath}
 with the $m_n$'s integers that are now called Markoff (or Markov) numbers. The first Markoff numbers are
   \begin{displaymath}1, 2, 5, 13, 29, 34, 89, 169, 194, 233, 433, 610, 985, 1325,\ldots \end{displaymath}
 To each Lagrange number ($<3$), there corresponds a finite number of explicit families of numbers (all having the same continued fraction expansion after some rank, for a given family) to be excluded, so that all other numbers satisfy the inequality
    \begin{displaymath}\left| x-\frac{p}{q}\right| < \frac{1}{L_n\,q^2}.\end{displaymath}  
 For more on this, see~\cite{Reutenauer}.

In trying to understand how to tie the hard to approximate property of the golden ratio to plant growth, we consider the following model. The ``plant'' is considered to be cylindrical, with buds growing successively on an upward helix at regular intervals (see Figure~\ref{fig:cylindric} and~\ref{reseau_cercles}). The length of these intervals is measured by the \define{divergence} $x$ in terms of the ``angle'' between two successive buds. This is expressed as a proportion of a complete turn (expressed in radians),  with the actual angle equal to $2x\pi$. It is stated in~\cite{Du,RZ}  that for best plant growth, $x$ must be not only irrational but in fact an irrational that is hardest as possible to approximate. 
Our purpose here is to exploit Markoff Theory to justify this last statement making use of a model suggested by Iterson~\cite[page 24]{iterson} that suggests what one could consider as an optimization parameter in plant growth. More explicitly, we consider a specific function $f(x,y)$ that measures how ``good'' a growth scheme is given by its divergence $x$, with $y$ denoting the height difference between successive buds. We show that $f(x,y)$ is ``globally optimal'' (that is for all $y$) if and only if $x$ is equivalent to the golden ratio.  From a mathematical perspective, the function $f(x,y)$ is both sound and with elegant properties. Noteworthy among these is the fact that it is invariant under the Modular Group, when considered as a function of the complex number $x+iy$. In fact this plays a key role in the proof of our main result. 

Further interesting mathematical work related to phyllotaxis may be found in the work of Adler \cite{A}, Atela, Gol\'e and Hotton \cite{AGH}, Coxeter \cite{C}, Leigh \cite{L}, Marzec and Kappraff \cite{MK}, Okabe \cite{Okabe1,Okabe2}; as well as in the papers collected in 
{\em Symmetry in Plants} \cite{JB}.

\section{A mathematical model based on the area around a bud}
As sketched above, we consider a spiral \define{growth scheme} on the cylinder to be specified by the pair of numbers $(x,y)$, with $x$ the divergence angle between successive buds, and $y$ the height difference between these buds, as illustrated in Figure~\ref{reseau_cercles}. 
To introduce a measure of how good a growth scheme $(x,y)$ is, Iterson suggested that one should surround each bud by the largest-area disk (pictured as spheres in Figure~\ref{reseau_cercles}, only for aesthetic reasons) so that no two disks overlap. Thus the diameter of these circles is the shortest possible distance between two buds. Heuristically put, one considers here that an optimal growth scheme for a plant would be to aim at sprouting the maximal number of buds with a minimal use of resources (here measured by disk-covering-area). Hence, for a given growth scheme, the proportion of area of the trunk covered by the aforementioned disks is considered to measure how capacious the growth scheme is.

\begin{figure}[ht]
\begin{center}
\includegraphics[scale=.7]{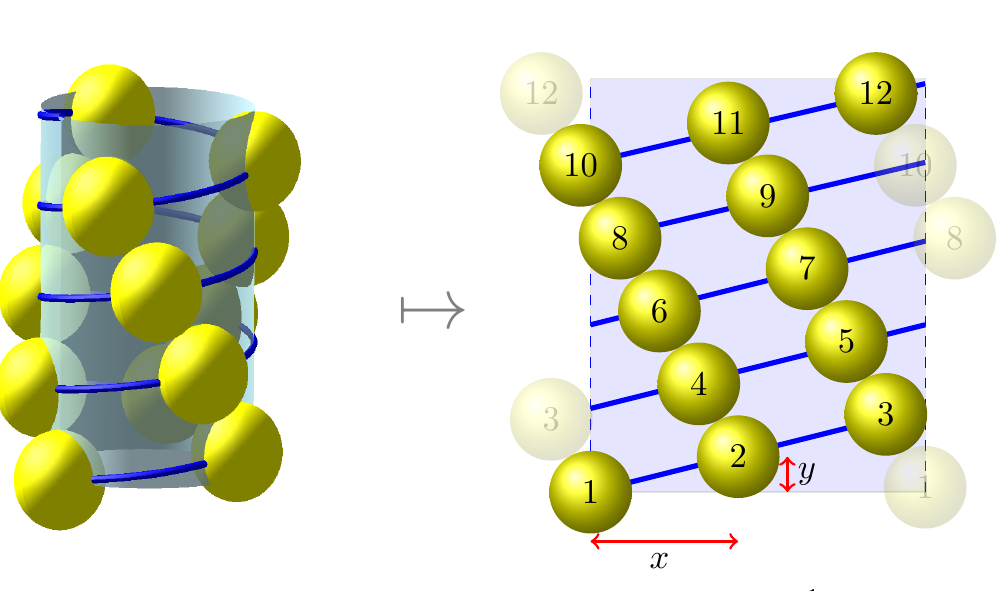}
\end{center}\vskip-10pt
\caption{Buds on a cylindrical trunc, and unfolded version.}  
\label{reseau_cercles}
\end{figure}

Unfolding the cylinder (and periodically repeating horizontally the pattern of buds) we get a lattice $\mathcal{L}_{xy}$ in the plane which is ``generated'' by the vectors $(1,0)$ (implicitly assuming that the circumference of the cylinder is equal to $1$), and $(x,y)$. More explicitly, we have
 \begin{displaymath}\mathcal{L}_{xy}:=\{\alpha\,(1,0)+\beta (x,y)\ |\ \alpha,\beta\in \Z\},\end{displaymath}
with buds placed at each points of $\mathcal{L}_{xy}$.  
\begin{wrapfigure}[7]{R}{0.35\textwidth}
\vskip25pt
\begin{picture}(30,30)(-30,7)
  \begin{tikzpicture}[scale=2.5]
  \draw[line width=1pt,color=blue,fill=blue,opacity=0.1] (0,0)-- (0,.835) -- (1.1,1.1) -- (1.1,.265)   -- cycle;
 \draw[line width=.6pt,color=blue] (0,0)-- (0,.835) -- (1.1,1.1) -- (1.1,.265)   -- cycle;
   \draw[line width=0pt,fill=yellow] (.55,.55) circle (.41);
   \draw (.55,.96) -- (.55,.14);
         \node at (.65,.65) {$d$};
   \draw[->,red,line width=1.5pt]  (0,-0.01) -- (0,.835) ;
      \node[red] at (-.1,.55) {$u$};
   \draw[->,red,line width=1.5pt]  (0,0) -- (1.1,.265);
      \node[red] at (.77,.1) {$v$};
\end{tikzpicture}
\end{picture}
\vskip10pt
\quad{Fig. 4 Disk inscribed in the fundamental region.}
\end{wrapfigure}
 \setcounter{figure}{4}Following Iterson, as mentioned above, we surround each point of $\mathcal{L}_{xy}$ by a disk whose diameter $d=d(x,y)$ is the smallest distance between two points of the lattice. 
The parallelogram with sides $u$ and $v$ (for any basis $u$, $v$ of $\mathcal{L}_{xy}$) is said to be a \define{fundamental region} for the lattice, and $\R\times \R$ is tiled by $\mathcal{L}_{xy}$ translates of this fundamental region. The area of said region is given by the absolute value of the determinant
whose row are the vectors $u$ and $v$. It is easy to see that this is equal to $y$. Indeed, this area does not depend on the choice of basis, hence we may choose the basis $\{(1,0),(x,y)\}$, and calculate the area as being
     $$\det\begin{pmatrix} 1 & 0\\ x& y \end{pmatrix}=y.$$
Up to a translation we may assume that the disks originally surrounding each point of $\mathcal{L}_{xy}$ are drawn with center in the middle of each of the translates of the fundamental region, as illustrated on the right. Thus the measure how well the disks cover the plane corresponds to the ratio of area of one of the disks (of radius $d(x,y)/2$) with respect to the area of one copy of the fundamental region, in formula this gives
${\pi\,d(x,y)^2}/(4\,y)$. 

Simplifying by a scalar multiple, we define the measure of ``capacity'' of a growth scheme as the quotient $d(x,y)^2/y$, considering as above that this capacity is directly correlated to the proportion of area covered by disks.  For a fixed divergence, we will study the behavior of the function $y\mapsto d^2/y$ and show, using Markoff Theory, that the minimum of this function is largest when $x$ is the golden ratio or an equivalent number. 

\section{Growth capacity is invariant under the modular group}

Let us first straightforwardly reformulate our construction above in terms of Poincar\'e's half-plane model of hyperbolic geometry, and its completion:
   \begin{displaymath}\H:=\{\omega \in \C\ ; \  \Im(\omega)>0\}, \qquad{\rm and}\qquad
      \overline{\H}:=\H \cup \R\cup \{\infty\}.\end{displaymath}
Each point $(x,y)$ (with $y>0$) is considered here as the point $\omega:=x+iy$ in $\H$. In this manner, we will consider points of $\H$  as encoding growth schemes. To each such growth scheme $\omega\in\H$, we associate the lattice $\mathcal{L}_\omega:=\mathbb Z + \mathbb Z\omega$. This is the additive subgroup of $\C$ generated by $1$ and $\omega$; and $d(\omega)$ is the minimal distance between two points of this lattice. Just as in our previous formulation, we have
    \begin{displaymath}d(\omega)=\mbox{min}\{\ |\alpha+\beta\,\omega|\ ;\  \alpha,\beta\in \mathbb Z, (\alpha,\beta)\neq (0,0)\}.\end{displaymath} 
Following our discussion of the previous section, we reformulate the \define{growth capacity} function  $f: \H\to \R$ as
\begin{equation}
   f(\omega):=\frac{d(\omega)^2}{\Im(\omega)}.
\end{equation}
It may very well be that this function has already been considered, together with Proposition~\ref{Prop:Invariance} below, but we could not find its trace in the  literature. 

We first recall basic facts about the action of the modular group $\PSL_2(\mathbb Z)$ on $\overline{\H}$. Elements of $\PSL_2(\mathbb Z)$ are $2\times 2$ matrices of determinant $1$ with coefficients in $\Z$, with $g$ identified with $-g$. The action $\PSL_2(\mathbb Z)\times \overline{\H}\rightarrow \overline{\H}$  is defined as
   \begin{displaymath}g\cdot \omega=\frac{a\,\omega+b}{c\,\omega+d},\qquad {\rm for}\qquad g=\begin{pmatrix} a&b\\c&d \end{pmatrix}\in \PSL_2(\mathbb Z),
       \end{displaymath}
with $g\cdot \infty:=a/c$ and $g\cdot (-d/c) =\infty$, when $c\not=0$; and $g\cdot \infty:=\infty$ otherwise. 
As is well-known, the modular group is generated\footnote{See for instance \cite{S}, Theorem 2 of chapiter VII.} by the two functions $T:\omega \mapsto \omega +1$, and $S:\omega\mapsto {-1}/{\omega}$, with relations 
   \begin{displaymath}S^2=\mathrm{Id},\qquad {\rm and}\qquad 
      (ST)^3=\mathrm{Id}. \end{displaymath}
\begin{figure}[ht]
\begin{center}
\includegraphics[scale=.7]{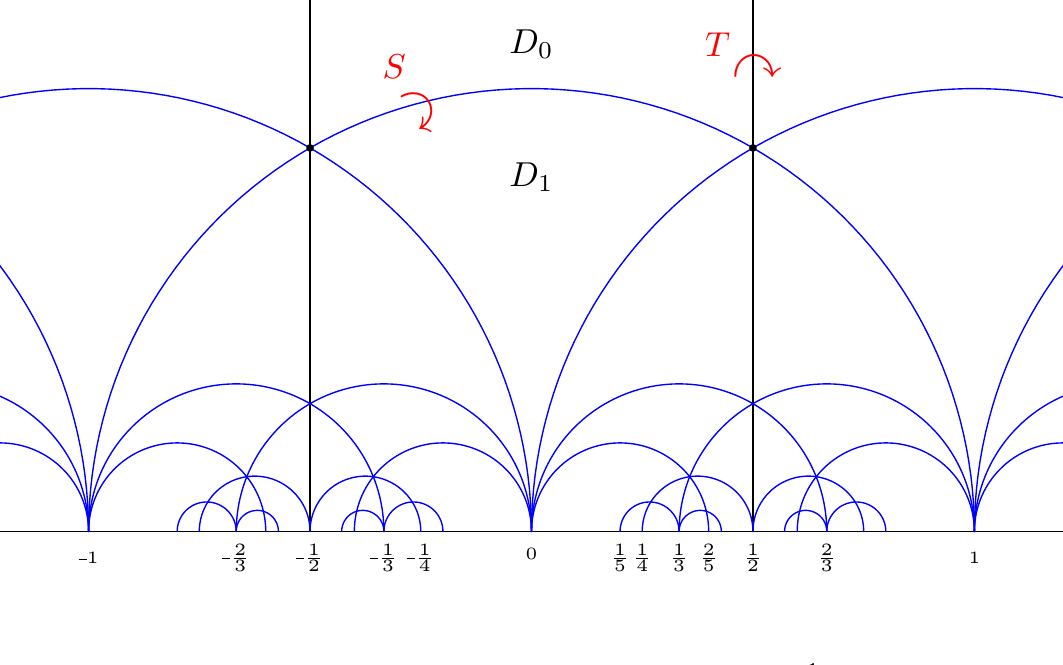}
\vskip-15pt
\caption{Tiling of hyperbolic plane.}  \label{pavagehyper}
\end{center}
\end{figure}
A very classical decomposition of the space $\H$, with respect to this action of the modular group, is obtained by considering all images under group elements of the fundamental region
    \begin{displaymath}D_0=\{\omega \in \C\ ;\  -{1}/{2}\leq\mbox{Re}(\omega)\leq{1}/{2},\quad {\rm and}\quad  |\omega|\geq 1\}.\end{displaymath}
 This results is a tiling of $\H$, partly shown in Figure~\ref{pavagehyper}, with $D_1$ being the image of $D_0$ under $S$ (which sends $\infty$ to $0$).

\begin{proposition}\label{Prop:Invariance}
The function $f$ is invariant under the modular group $\PSL_2(\mathbb Z)$, that is $f(g\cdot\omega)=f(\omega)$ for all $g\in \mbox{SL}_2(\mathbb Z)$ and $\omega\in\H$.
\end{proposition}
\begin{proof}
It is clearly sufficient to show that $f$ is invariant for $T$ and $S$. It is evident in the first case, since the lattice generated by $1$ and $\omega$ coincides with the lattice generated by  $1$ and $\omega+1$ on one hand; and on the other because the imaginary parts of $\omega$ and $\omega+1$ are equal.
The second case proceeds as follows. Observe first that elements of the lattice $\mathcal{L}(-1/\omega)$ may be written as multiples of $1/\omega$ by elements of $\mathcal{L}(\omega)$:
    \begin{displaymath}\alpha+\beta\left(\frac{-1}{\omega}\right)=\frac{1}{\omega}(\alpha\,\omega-\beta).\end{displaymath}
Hence, the module of  $\alpha+\beta\left({-1}/{\omega}\right)$ is equal to that of $\alpha\omega-\beta$ (which lies in $\mathcal{L}(\omega)$)  divided by $|\omega|$. Since this links all elements of $\mathcal{L}(-1/\omega)$ to a corresponding element of $\mathcal{L}(\omega)$, it follows that $d({-1}/{\omega})=d(\omega)/{|\omega|}$. On the other hand, 
  \begin{displaymath}\Im\left(\frac{-1}{\omega}\right)=\Im\left(\frac{-\bar{\omega}}{\omega\bar\omega}\right)=\frac{\Im(\omega)}{|\omega|^2}.\end{displaymath} 
Thus 
   \begin{displaymath}f\left(\frac{-1}{\omega}\right)=\frac{d(\omega)^2/|\omega|^2}{\Im(\omega)/|\omega|^2}=f(\omega),\end{displaymath}
which concludes the proof.
\end{proof}

\begin{proposition}\label{prop:valfond}
If $\omega$ lies in $D_0$ or any of its horizontal translates $D_0+n=T^n(D_0)$, for $n\in \mathbb Z$, then 
$f(\omega)={1}/{\Im(\omega)}$.
\end{proposition}

\begin{proof}
For $\omega=x+iy\in D_0$, elements of the lattice $\Z+\Z\omega$  are of the form $\alpha+\beta\omega=\alpha+\beta \,x+i\beta\,y$, and
\begin{eqnarray*}
    |\alpha+\beta \omega|^2&=&(\alpha+\beta\,x)^2+(\beta\,y)^2\\
         &=&\alpha^2+2x\,\alpha \beta +(x^2+y^2)\beta^2.
   \end{eqnarray*}
Note that $2\,|x|\leq 1$, so that $(-2)\,|x|\geq -1$, and we get 
\begin{eqnarray*}
   \alpha^2+2x\,\alpha \beta +(x^2+y^2)\beta^2&\geq& \alpha^2+(x^2+y^2)\beta^2-2|x\,\alpha \beta |\\
      &\geq& \alpha^2-|\alpha\beta |+\beta^2\\
      &=& |\alpha|^2-|\alpha||\beta |+|\beta |^2.
  \end{eqnarray*}
For $\alpha$ and $\beta$ in $\Z$, the quadratic form $\alpha^2-\alpha \beta+\beta^2$ only takes positive integral values, since its discriminant is $-3$. Its minimum value, for $\alpha,\beta $ not both $0$, is thus $1$. It follows that the minimum value of $|\alpha+\beta \omega|^2$, under the same conditions for $\alpha,\beta$, is also $1$. Thus we have shown that $d(\omega)^2=1$, and we get the announced formula for $f(\omega)$ in this case. When $\omega\in n+D_0$, the result also holds since both $f$ and the imaginary part of $\omega$ are invariant under horizontal translations. This completes our proof.
\end{proof}
The previous result implies that $f$ is bounded above by $2/\sqrt{3}$, since this is the maximal value of $f$ in the fundamental domain $D_0$.

\begin{corollary}
The function $f$ is continuous.
\end{corollary}

\begin{proof}
Clearly the restriction of $f$ to $D_0$ is continuous. For $g$ in the modular group, the restriction of $f$ to $g\cdot D_0$ is also continuous, 
since this restriction maps $\omega \in g\cdot D_0$ to
\begin{equation}\label{eq:formuleinvariant}
    f(\omega)=f(g^{-1}\omega)=\frac{1}{\Im(g^{-1}\omega)}
 \end{equation}
in view of the invariance of $f$ under the modular group, and by Proposition~\ref{prop:valfond}, knowing that $g^{-1}\cdot\omega\in D_0$.  But $g^{-1}$ is continuous, hence $f$ is continuous on $gD_0$.
We know that $\H$ is the union of the $gD_0$, for $g$ running over  $\PSL_2(\mathbb Z)$ (see \cite{S} Theorem 1 du chapitre VII). Moreover, at most three of these images contain any given point (see Figure \ref{pavagehyper}). It follows that $f$ is continuous at these finitely covered points, and $f$ is continuous everywhere. Thus showing the overall assertion.
\end{proof}

\section{Geometrical interpretation of growth capacity}
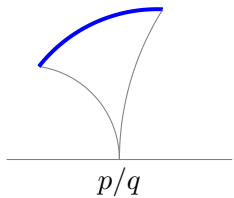
\begin{wrapfigure}[12]{L}{0.27\textwidth}
 \begin{tikzpicture}[scale=3]
 \segmentcol{2.083}{0.4167}{0}{1.427}{gray}
 \Segmentcol{2.667}{0.6667}{1.532}{2.475}{blue}
 \segmentcol{3.750}{1.250}{2.580}{3.142}{gray}
  \draw[color=gray] (2,0) -- (3,0);
\node at (2.5,-.1) {\small ${p}/{q}$};
 \end{tikzpicture}
 \caption{Cusp of triangle.}\label{fig:tiangle}
 \end{wrapfigure}
 Let us now consider how $f$ behaves for $\omega=x+iy\in\H$, with $x$ fixed. Proposition~\ref{prop:valfond}  takes care of all cases when $y>1$ (at least), and the interesting behavior is thus when $y$ becomes smaller and smaller. To better see this, we consider   $y={1}/{t}$, hence the function that sends $t$ to $f(x+{i}/{t})$. Figure~\ref{graphe}, illustrates how this function behaves for some fixed $x$. Once again we consider the tiling of $\H$ made out of the regions $g\cdot D_0$. Each of these is an \define{hyperbolic triangle}, with exactly one of its vertices in $\overline{\R}=\R \cup\{\infty\}$ (the regions $n+D_0$ are those for which this vertex is at $\infty$). This special vertex is said to be the \define{cusp} of the triangle and, except for the cases $n+D_0$,  it is located at some rational number $p/q$. The \define{basis} of the triangle is the edge opposite to the cusp.  See Fig.~6 above for an illustration of such a triangle and its cusp, with edge basis colored in blue (as is also the case in upcoming figures).
 
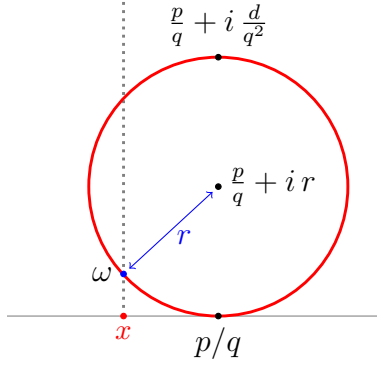
\begin{figure}[ht]
\begin{center}
\begin{tikzpicture}[scale=.7]
\clip(1,-1) rectangle (8,6);
\draw[color=gray] (0,0) -- (10,0);
\draw[color=gray,dotted,line width=1.1pt] (3.2,0) -- (3.2,8);
\draw[color=red,line width=1.1pt] (5,2.46) circle (70pt);
\node[anchor=west] at (5,2.46) {$\frac{p}{q}+i\,r$}; 
\draw[fill=black]  (5,2.46) circle (1.5pt);
\draw[fill=black]  (5,4.92) circle (1.5pt);
\node[anchor=south] at (5,4.92) {$\frac{p}{q}+i\,\frac{d}{q^2}$}; 
\draw[<->,shorten <=3pt,shorten >=3pt,color=blue]  (5,2.46) -- (3.2,.8);
\draw[color=blue,fill=blue] (3.2,.8) circle (1.5pt);
\node[anchor=east] at (3.2,.8) {$\omega$};
\node[anchor=north] at (5,0) {{${p}/{q}$}};
\node[color=blue,anchor=east] at (4.7,1.5) {{$r$}};
\draw[fill=black]  (5,0) circle (1.5pt);
\draw[color=red,fill=red] (3.2,0) circle (1.5pt);
\node[color=red,anchor=north] at (3.2,0) {{\small $x$}};
\end{tikzpicture}
\end{center}\vskip-10pt
\caption{Interpretation of $f(\omega)$ as a radius (up to a scalar multiple).}\label{fig:rayon}
\end{figure}
 Exploiting the propositions of the previous section, we may give an elegant geometrical interpretation of the function $f(\omega)$. 
 Indeed, it follows from Prop.~\ref{prop:valfond} that $f(\omega)$ is constant along an horizontal line $\Im(\omega)=1/d$, for $d\leq 1$, since the line is then entirely contained in the translates $D_0+n$ for $n\in\Z$. By general principles of inversive geometry, the image of this line under the modular group transformation 
    \begin{displaymath}\omega \mapsto \frac{p\,\omega+ p'}{q\,\omega +q'},\qquad {\rm for}\qquad g=\begin{pmatrix} p&p'\\q&q'\end{pmatrix}\in \PSL_2(\mathbb Z),\end{displaymath}
 is a circle tangent to the real axis at $p/q=g\cdot\infty$. Its radius is equal to $r=d/(2\,q^2)$, and hence its center is $p/q+i\,r$. 
 Indeed,  we have
    \begin{displaymath}g\cdot(x+i/d) = {\frac { \left( px+p' \right)  \left( qx+q' \right) {d}^{2}+pq}{ \left( 
qx+q' \right) ^{2}{d}^{2}+{q}^{2}}}+i\,{\frac {\,d}{ \left( qx+q' \right) ^{
2}{d}^{2}+{q}^{2}}}
\end{displaymath}
which evaluates to $p/q+i\,d/q^2$ at $x=-q'/q$. Since this is the point diametrically opposed to $p/q$, perforce the diameter of the circle is its $y$-coordinate, hence our formula. 

On the other hand, from Proposition~\ref{prop:valfond} we deduce that 
\begin{equation}\label{eq:formule_pour_f}
     f(x+i/t)=(x\,q-p)^2\,t+{q^2}/{t}.
  \end{equation}
by applying~\pref{eq:formuleinvariant} to $\omega=x+i/t$ in $g\cdot D_0$, using $pq'-qp'=1$, via the calculation
\begin{eqnarray*}
     f(x+i/t) &=& \frac{1}{\Im(g^{-1}(\omega))}\\
                &=& \Im\left(\frac{q'\,\omega-p'}{-q\,\omega+p} \right)^{-1}\\
                &=& \Im\left(\frac{(q'\,\omega-p')(-q\,\overline{\omega}+p)}{(-q\,\omega+p)(-q\,\overline{\omega}+p)} \right)^{-1}\\
                &=&\Im\left(\frac{(-q'\,q\,x^2+x+(p\,p'-q\,q'/t^2)+i/t}{(x\,q-p)^2+{q^2}/{t^2}}\right)^{-1}\\
                &=& (q\,x-p)^2\,t+{q^2}/{t},
  \end{eqnarray*}
as announced. As it happens, this last right-hand side affords the following simple geometrical interpretation.

\begin{proposition}\label{prop:rayon}
For any $\omega\in\H$, let $g\in \PSL_2(\mathbb Z)$ be such that $g^{-1}\cdot\omega$ lies in some $D_0+n$ (for $n\in\Z$), and let $p/q:=g\cdot\infty$. Then, the value of the function $f$ at $\omega$ is equal to $d\,q^2$, with $d$ being the diameter of the circle which is tangent to the real axis at $p/q$, and which passes through the point $\omega$.
\end{proposition}  

\begin{proof}
In terms of real coordinates, the equation of the circle considered (represented in Figure~\ref{fig:rayon}) is  $(x-p/q)^2+(y-r)^2-r^2=0$.
Multiplying both sides by $q^2/y$, this may be written as
  $(q\,x-p)^2/y+q^2y-d\,q^2=0$, with $d=2\,r$.
Thus, with $y=1/t$, we get
    \begin{displaymath}d\,q^2=(q\,x-p)^2t+q^2/t=f(x+i/t),\end{displaymath}
 thus showing our assertion.
\end{proof}

\begin{figure}[ht]
\begin{center}
\includegraphics[width=110mm]{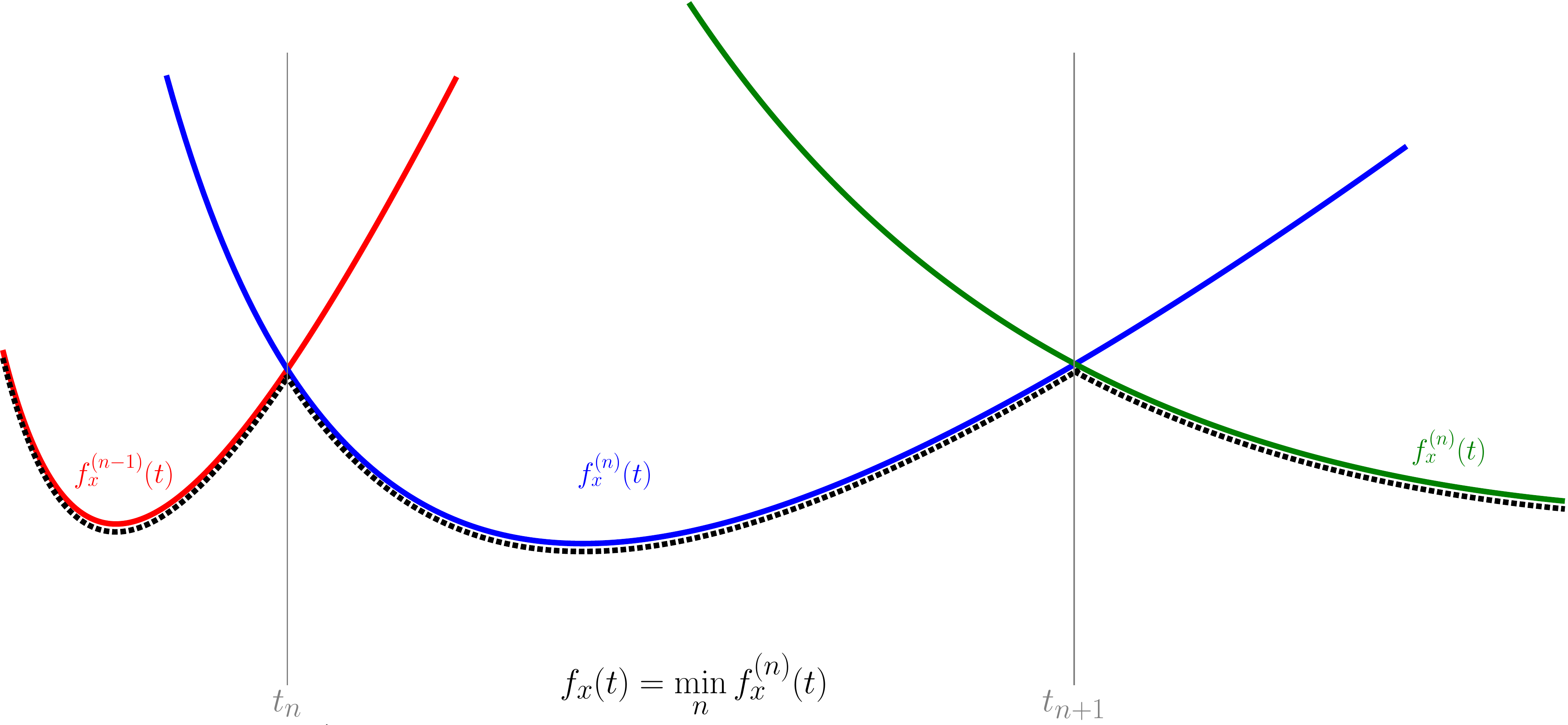}
\end{center}
\caption{The function $f_x$ is obtained by gluing pieces of successive functions $f_x^{(n)}(t)$.}  
\label{fig:minima}  
\end{figure}
The following proposition will help us tie our growth capacity measure to how well or not the number $x$ may be approximated by a rational number. To this end, we first clarify the domain on which formula~\pref{eq:formule_pour_f} applies, for a fixed value of $x$. Since the right-hand side of \pref{eq:formule_pour_f} is a smooth convex function of $t$, and $f$ is globaly continuous, it results that 
     \begin{displaymath}f_x:=t\mapsto f(x+i/t)\end{displaymath}
is a piecewise smooth convex function between some local maxima, where it is not derivable. More precisely, we have an increasing sequence of real numbers $t_n=t_n(x)$
      \begin{displaymath}t_1<t_2<\ \ldots\ <t_{n-1} <t_n<\ \ldots\end{displaymath}
 such that the function $f_x$ is (locally) given by the formula 
\begin{equation}\label{eq:formule_fn}
    f_x^{(n)}(t):=(x\,q_n-p_n)^2\,t+{q_n^2}/{t}.
 \end{equation}
This is to say that $f_x(t)=f_x^{(n)}(t)$, when $t_n\leq t\leq t_{n+1}$. Observe that $f_x^{(n)}(t)$ makes sense for all $t>0$, and Figure~\ref{fig:minima} illustrates how $f_x$ is obtained by gluing pieces of successive functions $f_x^{(n)}(t)$, for increasing values of $n$. 
We will see later that $p_n/q_n$ is the $n^{\rm th}$ Hermite convergent of $x$. This will imply that 
   \begin{displaymath} f_x^{(n-1)}(t)< f_x^{(n)}(t),\qquad {\rm if}\qquad t<t_n,\end{displaymath}
 and 
   \begin{displaymath} f_x^{(n-1)}(t)> f_x^{(n)}(t),\qquad {\rm if}\qquad t>t_n;\end{displaymath} 
 hence $t_n$ is a local maximum of $f_x$. We may thus write
     \begin{displaymath}f_x(t)=\min_n f_x^{(n)}(t),\end{displaymath}
with  the minimum taken over $n$, for any fixed $t$.
Continuity of $f$ forces $f_x^{(n)}$ to agree with $f_x^{(n-1)}$ at $t_n=t_n(x)$, hence
   \begin{displaymath}(q_n\,x-p_n)^2\,t_n+{q_n^2}/{t_n}=
   (q_{n-1}\,x-p_{n-1})^2\,t_n+{q_{n-1}^2}/{t_n}.\end{displaymath}
Solving this equality for $t_n$ gives \begin{equation}\label{eq:formule_maxima}
   t_n:=\sqrt{\frac{  q_{n}^{2} -q_{n-1}^{2}}{ \left( q_{n-1}\,x-p_{n-1} \right) ^{2}- \left( x\,q_{n}-p_{n} \right) ^{2}}},
\end{equation}
and we may then calculate directly that
   \begin{displaymath}f_x(t_n)=t_n\,\frac{p_n/q_n+p_{n-1}/q_{n-1}-2\,x}{q_{n}/q_{n-1} -q_{n-1}/q_n}.\end{displaymath}

\begin{proposition}\label{prop:formule}
The local minima of the function $f_x$, from $\R^*$ to $\R$, are the numbers $2\,|q_n(q_n\,x-p_n)|$,
 and these are achieved at $t_0=|q_n/(q_n\,x-p_n)|$.
\end{proposition}

\begin{proof} Assume that $f_x$ is given by formula~\pref{eq:formule_fn} in the segment $t_n\leq t\leq t_{n+1}$, and observing that this is a convex function, therefore the minimum occurs when
   \begin{displaymath}\frac{d}{dt} f_x^{(n)} = -{q_n^2}/{t^2}+(q_n\,x-p_n)^2=0,\end{displaymath}
hence when $t$ is equal to $t_0:=\left| q_n/(q_n\,x-p_n)\right|$, and the corresponding value
   \begin{displaymath}f_x^{(n)}(t_0) = 2\,|q_n(q_n\,x-p_n)|\end{displaymath}
is the announced minimum. 
\end{proof}
Geometrically, this minimum occurs when the circle of Proposition~\ref{prop:rayon} is tangent to the vertical line whose points have real part equal to $x$.  

\section{Global behavior of growth capacity}

We will now see that the global behavior of $f_x$ may be revealed using interesting properties of  Hermite's {\em approximation theory} \cite{H} for real numbers. We will exploit this to understand what singles out the golden ratio as a champion from the point of view of the associated growth capacity function.
\begin{figure}[ht]
\begin{center}
\includegraphics[scale=.7]{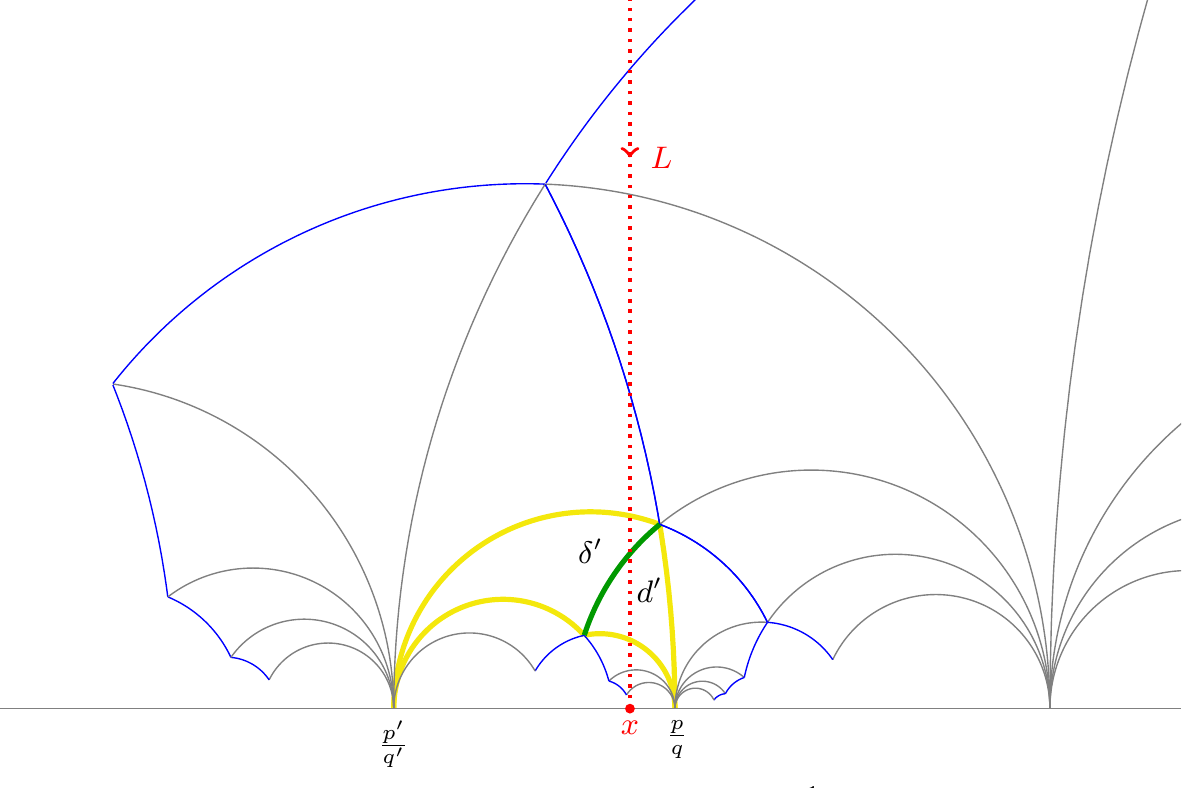}
\end{center}
\caption{Traveling down from infinity along the line $\Re(\omega)=x$.}  
\label{fig:cusp}  
\end{figure}

To this end, we borrow on Humbert's approach (see \cite{Hu}) to Hermite's theory. 
Consider a point traveling down a vertical hyperbolic line of abscissa $x$, going from $\infty$ to $0$. In other words, these are the form $x+{i}/{t}$, with $t$ going from $0$ to $\infty$. The point successively traverses 
hyperbolic triangles $g\cdot D_0$ as illustrated in Figure~\ref{fig:cusp}, whose cusps (at $p/q\neq \infty$) are by definition (\cite[page 82]{Hu} , or \cite{jacobs})  \define{Hermite convergents} of $x$. These convergents satisfy
   \begin{displaymath}\left| x-\frac{p}{q}\right|\leq \frac{1}{\sqrt{3}\,q^2},\end{displaymath}
 and the two vertices lying on the basis of the hyperbolic triangles in question are on the (real plane) circle of equation 
\begin{equation}\label{eq:cercle}
   (x-p/q)^2+(y-r)^2=r^2,\qquad {\rm with}\qquad r=\frac{1}{\sqrt{3}\,q^2}.
   \end{equation}
Let $x=[a_0,a_a,a_2,\cdots]$ be the continued fraction expansion of $x$ (positive), with $a_i\in \N$. 
Recall that its (classical) \define{convergents} are the rational numbers
   \begin{displaymath}\frac{p_n}{q_n}=[a_0,a_1,\cdots, a_n]= 
a_0+\cfrac{1}{a_1+\cfrac{1}{a_2+\cfrac{1}{\ddots\ +\cfrac{1}{a_n}}}}.
\end{displaymath}
 for $n\in\N$. Successive Hermite convergents appear as a subsequence of the sequence of all classical convergents of $x$, see \cite[page 95]{Hu}. One property that characterizes some of the Hermite convergents of $x$ goes as follows. If  $a_{n+1}\geq 2$, then $[a_0,a_1,\cdots, a_n]$ is an Hermite convergent of $x$, see \cite[page 96]{Hu}.
Moreover, if ${p'}/{q'}$ and ${p}/{q}$ are two consecutive Hermite convergents, then $p'q-pq'=\pm 1$, see \cite[page 84]{Hu}. 
This is not an exhaustive set of property if one intends to characterized the $p_n/q_n$, and we refer to {\em loc. cit.} for the necessary details.
Just to illustrate, the first convergents of $\sqrt{7}-1$ are:
    \begin{displaymath}\frac{2}{1},\ \frac{3}{2},\ \frac{5}{3},\ {\frac{23}{14}},{\frac{28}{17}},\ {\frac{51}{31}},{\frac{79}{
48}},\ {\frac{367}{223}},\ {\frac{446}{271}},\ {\frac{813}{494}}
,\ \ldots\end{displaymath}
 whereas among these the only Hermite convergents are:
      \begin{displaymath}\frac{2}{1},\ \frac{5}{3},\ {\frac{28}{17}},\ {\frac{79}{48}},\ {\rm and}\ {\frac{446}{271}}.\end{displaymath}

\begin{figure}[ht]
\begin{center}
\begin{tikzpicture}[scale=.7]
  \node (spirale) at (0,0)
    {\includegraphics[width=100mm]{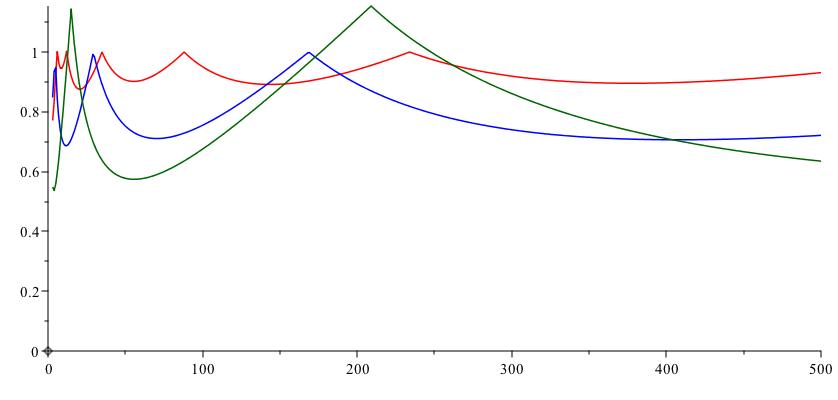}};
 \node[color=red,anchor=west] at (7,2.3) {\tiny $\ f_{\varphi}$};
 \node[color=blue,anchor=west] at (7,1.3) {\tiny $f_{\sqrt{2}-1}$};
 \node[color=green!50!black,anchor=west] at (7,.5) {\tiny $f_{\sqrt{3}-1}$};
 \node at(7.3,-2.8) {\tiny {$N$}};
 \end{tikzpicture}
\end{center}
\caption{Three growth capacity functions.}  
\label{graphe}
\end{figure}
Figure~\ref{graphe} illustrates that a local minimum occurs in each of the region associated to an Hermite convergent (in which $f_x(t)$ may be calculated using formula~\pref{eq:formule_pour_f}), with three different values of $x$.

We now want to establish that the golden ratio (and equivalent numbers) gives the  best growth scheme. 
As before, for a real $x$, let $x=[a_0,a_1,a_2,\ldots]$ be its continued fraction expansion. We assume that this expansion is infinite, which is to say that $x$ is irrational. Once again denote by $[a_0,\ldots,a_n]= {p_n}/{q_n}$ its $n$-th convergent, expressed as an irreducible fraction. It is well known \cite[(1.15) section 1.4]{A}  that
for all  $n\geq 1$, we have
    \begin{displaymath}\left|x-\frac{p_n}{q_n}\right|=\frac{1}{\lambda_n(x)\,q_n^2},\end{displaymath}
with $\lambda_n(x)=[a_n,\ldots,a_1]^{-1}+[a_{n+1},a_{n+2},\ldots]$.
Equivalently, $|q_n(q_n\,x-p_n)|={1}/{\lambda_n(x)}$. Moreover, the supremum of the  $\lambda_n(x)$, as $n$ goes to $\infty$, is precisely the \define{Lagrange number} of $x$ mentioned earlier, and it is denoted by $L(x)$, see \cite[(1.15), Proposition 1.22 and Definition 1.7]{A}. From Markoff's Theory, we know that for $x$ equal to the golden ratio, or any number whose continued fraction expansion contains only $1$ starting from some rank, then $L(x)=\sqrt{5}$; and that for any other number, $L(x)\geq \sqrt{8}$ ({\em loc.cit.}). From this we get the following, after proving an auxiliary lemma.

\begin{theorem}\label{thm:principal}
If $x$ is equal to the golden ratio, or any number whose continued fraction expansion contains only $1$ starting from some rank, then the supremum of the minima of its growth capacity function is ${2}/{\sqrt{5}}$. For any other number $x$, this limit is $\leq {2}/{\sqrt{8}}$.
\end{theorem}

For a given $x$, let us denote by $H(x)$ the subset of integers $n$ such that  ${p_n}/{q_n}$ is an Hermite convergents for $x$. For instance, for $x=\sqrt{7}-1$, we have
     \begin{displaymath}H(x)=\{ 0,2,4,6,8,\ldots \}.\end{displaymath}

\begin{lemma}
The supremum as $n$ goes to infinity of the sequence of all $\lambda_n(x)$, for $n\in N$, is equal to the supremum of the subsequence $(\lambda_n(x))_{n\in H(x)}$.
\end{lemma}

\begin{proof}
Let ${p}/{q}$ be an irreducible fraction, with $q>0$. Let us set $u=\varepsilon\,q\,(p-q\,x)$ where $\varepsilon=\pm 1$ is chosen so that $u$ is positive. Consider $q'$ the unique integer solution of $p\,q'\equiv \varepsilon\ \mathrm{mod}\ q$ with $0\leq q'<q$. Let $p'$ be such that $p\,q'=\varepsilon +q\,p'$. Then ${p}/{q}$ is an Hermite convergent for $x$ if and only if    
           \begin{displaymath}u<\frac{q(q+2q')}{2(q^2+qq'+q'^2)},\end{displaymath}
see \cite[page 95]{Hu}. Observe that, since $q>q'$, we have 
   \begin{displaymath}\frac{q(q+2q')}{2(q^2+q\,q'+q'^2)}>\frac{q(q+2q')}{2(q^2+q\,q'+q\,q')}=\frac{1}{2}.\end{displaymath} 
It follows that a convergent ${p_n}/{q_n}$ which is not an Hermite convergent, must be such that $|q_n(q_n\,x-p_n)|=u>{1}/{2}$, and hence $\lambda_n(x)<2$. Since the supremum of $\lambda_n(x)$ is greater or equal to $\sqrt{5}/2$, it follows that this limit does not change if we restrict $n$ to be such that $ {p_n}/{q_n}$ is an Hermite convergent, that is $n\in H(x)$.
\end{proof}

We can now prove the theorem as follows.

\begin{proof}[\bf Proof of Theorem~\ref{thm:principal}]
Proposition~\ref{prop:formule} says that the minima are of the form $2|q(q\,x-p)|$, where ${p}/{q}$ is an Hermite convergent for $x$. This Hermite convergent occurs as one of the convergents of the continued fraction of $x$, say ${p}/{q}={p_n}/{q_n}$. By the above formula, this minimum is of the form ${2}/{\lambda_n(x)}$, where $n$ is the rank of an Hermite convergent, i.e.: $n\in H(x)$. By the lemma, the supremum of these numbers is ${2}/{L(x)}$, and the corollary follows.
\end{proof}

\section{Further considerations}
As we have seen, in instances where growth capacity could be considered to be a good measure from the point of view of phyllotaxis, it gives a clear mathematical indication why one should so often encounter the golden ratio. The theory considered here also suggests that if other growth schemes could occur in exceptional (or extraterrestrial!) instances, then the next most frequent such growth schemes would be tied to the number $1+\sqrt{2}$ (and equivalents); with variants of the Pell numbers, $P_n$,  
   \begin{displaymath}1, 2, 5, 12, 29, 70, 169, 408, 985, 2378, 5741, 13860, 33461, 80782, \ldots \end{displaymath}
replacing the Fibonacci numbers (and their own variants). After that would come, in rarer and rarer instances, growth schemes associated to the numbers
   \begin{displaymath}\frac{11+\sqrt{221}}{10},\ \frac{29+\sqrt{1517}}{26},\ \cdots\end{displaymath}
  For more on this from the point of view of Markoff Theory, see \cite[Section 10.2]{Reutenauer}.

Our explanation of the optimality of the golden ratio may be seen to be even more plausible if one considers the average 
\begin{equation}\label{eq:moyenne}
   g_x:=\limsup_{n\to\infty}  \overline{f}_x(n), \qquad \hbox{with}\qquad  \overline{f}_x(n):=\frac{1}{t_{n+1}-t_{n}}\int_{t_{n}}^{t_n+1} f_x^{(n)}(t)\,dt,
 \end{equation}
as a comparison tool between growth schemes. 
Rather than only whining from the point of view of a local behavior of minima,  $g_x$ gives a global measure that may be even more significant from the biological point of view. For the golden ratio $\varphi$, we observe that $2/\sqrt{5}\approx 0.89443\ (<g_\varphi)$ is   an upper bound for $g_x$, for all $x$ not equivalent to $\varphi$. More technically, it may be shown (see appendix)  that
\begin{eqnarray*}
    g_\varphi&=&\frac{1}{2}  +\frac{2}{\sqrt{5}}\log(\varphi)\\
                  &\approx& 0.93041,
\end{eqnarray*}
and that $ g_x< g_\varphi$ for all number $x$ not equivalent to $\varphi$. For instance,
\begin{eqnarray*}
    g_{(1+\sqrt{2})} &=& \frac{1}{2}  +\frac{1}{\sqrt{8}}\log(1+\sqrt{2})\\
                  &\approx& 0.81161.
\end{eqnarray*}   


\section*{acknowledgements}
We would like to thank St\'ephane Durand (see \cite{Du}) and Christiane Rousseau (see \cite{RZ}) for drawing our attention to the fact that: it is because it is hard to approximate by rational numbers that the golden ratio plays a key role in phyllotaxis. We also thank 
 Nadia Lafreni\`ere and Caroline Series for interesting suggestions.

\section*{Appendix}
We calculate here the limit of the averaging integral~\pref{eq:moyenne} in the case of the golden ratio $x=\varphi$, whose Hermite's convergents are the quotients $F_{n+1}/F_{n}$. The required calculation is greatly simplified using the simplified expression
\begin{equation}\label{eq:fxn_golden}
    f_\varphi^{(n)}(t)=\varphi^{-2\,(n+1)}t+F_n^2\,t.
 \end{equation}
 for the function $f_\varphi^{(n)}$, which follows from the fact that
$(F_{n}-F_{n+1}\,\varphi)^2=\varphi^{-2\,(n+1)}$.
The corresponding minimum occurs at $F_n\,\varphi^{n+1}$,
and takes the value 
     $$2F_n/\varphi^{n+1}\approx 2/\sqrt{5}.$$
These assertions follows from the well known Binet formula
   \begin{displaymath}F_n=\frac{\varphi^{n+1}-(-1/\varphi)^{n+1}}{\sqrt{5}}. 
   \end{displaymath}
Exploiting that, $0\leq \left|{(-1/\varphi)^{n+1}}\right| \ll 1$, we also deduce from it the
very good approximation $F_n\approx{\varphi^{n+1}}/{\sqrt{5}}$. 
Thus $ f_\varphi^{(n)}(t)$
is very well approximated by $\varphi^{-2\,(n+1)}\,t+{\varphi^{2\,(n+1)}}/{(5\,t)}$ when $n$ is large enough.
We may also calculate that 
\begin{eqnarray*}
     t_n(\varphi) &=&\sqrt{\frac{F_n^2-F_{n-1}^2}{\varphi^{-2\,n}-\varphi^{-2\,(n+1)}}} \\
        &=&\varphi^{n}\,F_{n-1}\,\sqrt{((F_n/F_{n-1})^2-1)\,\varphi}\\
        &\approx& \varphi^{n+1}F_{n-1}\approx \varphi^{2\,(n+1)}/\sqrt{5},
   \end{eqnarray*}
from which we get
   \begin{displaymath}\frac{t_{n+1}^2-t_n^2}{t_{n+1}-t_{n}}=t_{n+1}+t_{n}\approx \varphi^{n+1}(\varphi\,F_{n}+F_{n-1})=
      \varphi^{2\,(n+1)},\end{displaymath}
 as well as
    \begin{displaymath}\frac{\log(t_{n+1}/t_n)}{t_{n+1}-t_{n}}\approx \sqrt{5}\,\frac{\log(\varphi^{2\,n+3}/\varphi^{2\,n+1})}{\varphi^{2\,n+3}-\varphi^{2\,n+1}}=2\,\sqrt{5}\,\frac{ \log(\varphi)}{\varphi^{2\,(n+1)}}\end{displaymath}
Hence, applying formula~\pref{eq:moyenne}, we find that
 \begin{eqnarray*}
   g_\varphi&=&\limsup_{n\to\infty}\frac{1}{t_{n+1}-t_{n}}\int_{t_{n}}^{t_{n+1}} f_\varphi^{(n)}(t)\, dt   \\
      &=&\limsup_{n\to\infty}  \frac{1}{t_{n+1}-t_{n}}\left[ \varphi^{-2\,(n+1)}\,\frac{t^2}{2} +
             \frac{\varphi^{2\,(n+1)}}{5}\log(t) \right]_{t=t_{n}}^{t=t_{n+1}}  \\ 
      &=&\limsup_{n\to\infty}     \frac{\varphi^{-2\,(n+1)}}{2}\,\left(\frac{t_{n+1}^2-t_n^2}{t_{n+1}-t_{n}}\right) +
             \frac{\varphi^{2\,(n+1)}}{5}\left(\frac{\log(t_{n+1}/t_n)}{t_{n+1}-t_{n}}  \right)  \\ 
       &=&\frac{1}{2}  +\frac{2}{\sqrt{5}}\log(\varphi).
  \end{eqnarray*}
 In the case of $\psi:=1+\sqrt{2}$,  one replaces Fibonacci numbers by Pell numbers, $P_n$, and uses
   \begin{displaymath}f_\psi^{(n)}(t)= {P_n}/{t} +\psi^{-2\,n}\,t,\qquad P_n\approx{\varphi^n}/{\sqrt{8}},\qquad{\rm and}\qquad
      t_n\approx{\psi^{2\,n+1}}/{\sqrt{8}},\end{displaymath}
to show that  $g_\psi={1}/{2}  +\log(\psi)/{\sqrt{8}}$ with a very similar calculation.



\end{document}